\documentclass[12pt]{amsart}
\usepackage[english]{babel}
\usepackage{amsmath}
\usepackage{graphicx}
\usepackage[colorlinks=true, allcolors=blue]{hyperref}
\usepackage[numbers]{natbib}
\usepackage{enumerate}
\usepackage{amsthm,amssymb,amsfonts,mathrsfs,xfrac,faktor,esint}
\usepackage{tikz-cd}
\usepackage{dsfont}
\usepackage{multirow}
\usepackage{tabularx}
\usepackage{xcolor}

\newtheorem{theorem}{Theorem}[section]
\newtheorem{coro}[theorem]{Corollary}
\newtheorem{lemma}[theorem]{Lemma}

\theoremstyle{definition}

\theoremstyle{remark}
\newtheorem{remark}[theorem]{Remark}

\usepackage{forloop}
\newcommand{\defcal}[1]{\expandafter\newcommand\csname 
	c#1\endcsname{{\mathcal{#1}}}}
\newcommand{\defbb}[1]{\expandafter\newcommand\csname 
	b#1\endcsname{{\mathds{#1}}}}
\newcommand{\defbf}[1]{\expandafter\newcommand\csname 
	bf#1\endcsname{{\mathbf{#1}}}}
\newcommand{\deffr}[1]{\expandafter\newcommand\csname 
	frak#1\endcsname{{\mathfrak{#1}}}}
\newcommand{\defov}[1]{\expandafter\newcommand\csname 
	ov#1\endcsname{{\overline{#1}}}}
\newcommand{\deftil}[1]{\expandafter\newcommand\csname 
	til#1\endcsname{{\widetilde{#1}}}}
\newcommand{\defhat}[1]{\expandafter\newcommand\csname 
	hat#1\endcsname{{\widehat{#1}}}}
\newcommand{\defscr}[1]{\expandafter\newcommand\csname 
	scr#1\endcsname{{\mathscr{#1}}}}
\newcommand{\deftt}[1]{\expandafter\newcommand\csname 
	tt#1\endcsname{{\mathtt{#1}}}}
\newcommand{\deful}[1]{\expandafter\newcommand\csname 
	ul#1\endcsname{{\underline{#1}}}}
\newcommand{\defas}[1]{\expandafter\newcommand\csname 
	'#1\endcsname{{\'{#1}}}}
\newcommand{\defde}[1]{\expandafter\newcommand\csname 
	`#1\endcsname{{\`{#1}}}}
\newcommand{\defpl}[1]{\expandafter\newcommand\csname 
	^#1\endcsname{{\^{#1}}}}
\newcommand{\defrm}[1]{\expandafter\newcommand\csname 
	rm#1\endcsname{{\mathrm{#1}}}}
\newcommand{\defdot}[1]{\expandafter\newcommand\csname 
	Dot#1\endcsname{{\Dot{#1}}}}
\newcommand{\defvec}[1]{\expandafter\newcommand\csname 
	Vec#1\endcsname{{\overrightarrow{#1}}}}

\newcounter{calBbCounter}
\forLoop{1}{26}{calBbCounter}{
	\edef\letter{\Alph{calBbCounter}}
	\expandafter\defcal\letter
	\expandafter\defbb\letter
	\expandafter\defbf\letter
	\expandafter\deffr\letter
	\expandafter\defov\letter
	\expandafter\deftil\letter
	\expandafter\defhat\letter
	\expandafter\defscr\letter
	\expandafter\deftt\letter
	\expandafter\deful\letter
	\expandafter\defas\letter
	\expandafter\defde\letter
	\expandafter\defpl\letter
	\expandafter\defrm\letter
    \expandafter\defdot\letter
    \expandafter\defvec\letter
}
\forLoop{1}{26}{calBbCounter}{
	\edef\letter{\alph{calBbCounter}}
	\expandafter\defbf\letter
	\expandafter\deffr\letter
	\expandafter\defov\letter
	\expandafter\deftil\letter
	\expandafter\defhat\letter
	\expandafter\defscr\letter
	\expandafter\deftt\letter
	\expandafter\deful\letter
	\expandafter\defas\letter
	\expandafter\defde\letter
	\expandafter\defpl\letter
	\expandafter\defrm\letter
    \expandafter\defdot\letter
    \expandafter\defvec\letter
}

\DeclareMathOperator{\dist}{dist}

\DeclareMathOperator{\Div}{div}
\DeclareMathOperator{\supp}{supp}

\newcommand{\Cc}{C^\infty_c}

\DeclareMathOperator{\bmo}{BMO}

\title{On representation of solutions to the heat equation}
\author{Pascal Auscher}
\address{Universit{\'e} Paris-Saclay, CNRS, Laboratoire de Math\'{e}matiques d'Orsay, 91405 Orsay, France}
\email{pascal.auscher@universite-paris-saclay.fr}

\author{Hedong Hou}
\address{Universit{\'e} Paris-Saclay, CNRS, Laboratoire de Math\'{e}matiques d'Orsay, 91405 Orsay, France}
\email{hedong.hou@universite-paris-saclay.fr}

\date{October 26, 2023}

\begin{document}
\maketitle

\begin{abstract}
    We propose a simple method to obtain semigroup representation of solutions to the heat equation using a local $L^2$ condition with prescribed growth and a boundedness condition within tempered distributions. This applies to many functional settings and, as an example, we consider the Koch and Tataru space related to $\bmo^{-1}$ initial data.
\end{abstract}
\section{Introduction}

The purpose of this note is to investigate representation for  solutions to the heat equation 
\begin{equation}
    \partial_t u - \Delta u = 0
    \label{e:heat}
\end{equation}
on the upper-half space $\bR^{1+n}_+:=(0,\infty) \times \bR^n$ or on a strip $(0,T) \times \bR^n$. That is, when can we assert that $u$ can be represented by the heat semigroup acting on a data, \textit{i.e.}, 
\begin{equation}
    u(t):=u(t,\cdot)=e^{t\Delta} u_0
    \label{e:representation}
\end{equation}
for some $u_0$ and all $t \in (0,T)$? 

The topic is not new, of course, so let us first briefly comment on some classical results in the literature.

The most general framework for such a representation is via tempered distributions. More precisely, given $u_0 \in \scrS'(\bR^n)$, then $t \mapsto e^{t\Delta} u_0$ lies in $C^\infty([0,\infty); \scrS'(\bR^n))$. Conversely, it has been shown in \cite[Chap.~3, Prop.~5.1]{Taylor2011PDE} that any $u \in C^\infty([0,\infty); \scrS'(\bR^n))$ solving the heat equation is represented by the heat semigroup applied to its initial value. Certainly, the argument still works in $C^1((0,\infty);\scrS'(\bR^n)) \cap C([0,\infty);\scrS'(\bR^n))$, which seems to close the topic. But it uses Fourier transform, so it is not transposable to more general equations (\textit{e.g.}, parabolic equations with coefficients). Thus, one may wonder whether different concrete knowledge, like a growth condition, on the solution could lead to a representation, not using Fourier transform. Yet, one can observe that growth exceeding the inverse of a Gaussian when $|x|\to\infty$ is forbidden for the representation.

Another framework is that of non-negative solutions. A classical result by D. Widder \cite[Thm.~6]{Widder1944Positive} shows that in one-dimensional case, any non-negative $C^2$-solution $u$ in the strip must be of the form \eqref{e:representation} for some  non-negative Borel measure $u_0$. It has been generalized to higher dimensions and classical solutions of parabolic equations with smooth coefficients by M. Krzyzanski \cite{Krzyzanski1964_Non-negative_nD}, via internal representation and a limiting argument. We are also going to use this idea below, but we want to remove the sign condition. D.G. Aronson later extended it to non-negative weak solutions of real parabolic equations, see \cite[Thm.~11]{Aronson1968Non-negative}.

Next, the uniqueness problem is tied with representation but they are different issues. For instance, let us mention two works giving sufficient criteria on strips for uniqueness, one by S. T\"acklind \cite{Tacklind1936_Class} providing the optimal pointwise growth condition, and the other by A. Gushchin \cite{Gushchin1984_Class} providing a local $L^2$ condition with prescribed growth, also optimal but more amenable to more general equations. In these results, the growth can be faster than the inverse of a Gaussian when $|x|\to\infty$, which hence excludes usage of tempered distributions, so uniqueness can hold without being able to represent general solutions. 

With these observations in mind, it seems that we have two very different theories to approach representation (and uniqueness): one only using distributions and Fourier transform; one not using them at all. The goal of this note is to make a bridge between them, \textit{i.e.}, to obtain tempered distributions, not just measurable functions or measures, as initial data from local integrability conditions. Such conditions may only include  integrability conditions in the interior, completed by a uniform control.

Let us state our result. A sequence $(T_k)$ of tempered distributions is  \textit{bounded} if $(\langle T_k,\varphi \rangle)$ is bounded for any $\varphi \in \scrS(\bR^n)$. Recall that any distributional solution to the heat equation on strips is in fact smooth by hypoellipticity, see for instance \cite[\S4.4]{Hormander2003_PdOI}.
\begin{theorem}
    \label{thm:rep} 
    Let $0<T \le \infty$. Let $u \in \scrD'((0,T)\times \bR^n)$ be a distributional solution to the heat equation. Suppose that:
    \begin{enumerate}[(i)]
        \item 
        \label{item:cond_size}
        (Size condition) For $0<a<b<T$, there exist $C(a,b)>0$ and $0<\gamma<1/4$ such that for any $R>0$,
        \begin{equation}
            \left( \int_a^b \int_{B(0,R)} |u(t,x)|^2 dtdx \right)^{1/2} \leq C(a,b) \exp\left( \frac{\gamma R^2}{b-a} \right);
            \label{e:L2condition}
        \end{equation}
        
        \item 
        \label{item:cond_initial}
        (Uniform control) There exists a sequence $(t_k)$ tending to 0 such that $(u(t_k))$ is bounded in $\scrS'(\bR^n)$.
    \end{enumerate}
    Then there exists a unique $u_0 \in \scrS'(\bR^n)$ so that $u(t) = e^{t\Delta}u_0$ for all $0<t<T$, where the heat semigroup is understood in the sense of tempered distributions.
\end{theorem}

Let us first give some remarks. The $L^2$ condition \eqref{item:cond_size} is only assumed on interior strips, and its growth is in the order of the inverse of a Gaussian. It can be proved that this condition alone implies that $u(t)$ is a tempered distribution for each $t>0$. The precise behavior of $C(a,b)$ is not required, and it could blow up as $a\to 0$. Condition \eqref{item:cond_initial} provides us with necessary uniformity to define the initial data, which \textit{a priori} does not follow from \eqref{item:cond_size}.

The conclusion also shows that $u$ has a natural extension to a solution of the heat equation on $(0,\infty)\times \bR^n$ if $T<\infty$. And if $T$ was already $\infty$, then we could get a control on $u$ when $t\to \infty$ via any knowledge we might get on $u_0$, \textit{e.g.}, if $u_0\in L^2(\bR^n)$ then $u(t)$ is bounded in $L^2(\bR^n)$.

As an interesting remark, the argument is not using Fourier transform at all, and it indeed extends with the same method of proof to more general parabolic equations with bounded measurable, real or complex coefficients at the expense of working in more restrictive spaces than $\scrS'(\bR^n)$. The reader can refer to \cite{Auscher-Monniaux-Portal2019existence}, \cite{Zaton2020}, and \cite{Auscher-Hou2023_SIO} for the proof and applications of the general result in the context of measurable initial data. Applications towards distributional data for more general equations will appear in the forthcoming work \cite{Auscher-Hou2023_HCLions}.

The consequence for uniqueness is as follows.
\begin{coro}
    \label{cor:unique} 
    Let $0<T \le \infty$. Let $u \in \scrD'((0,T)\times \bR^n)$ be a distributional solution to the heat equation. Suppose that \eqref{item:cond_size} holds and that there exists a sequence $(t_k)$ tending to 0 such that $(u(t_k))$ converges to 0 in $\scrS'(\bR^n)$. Then $u=0$.
\end{coro}

\begin{remark}
    \label{rem:initial_cond}
    Convergence in $\scrS'(\bR^n)$  cannot be replaced by convergence in $\scrD'(\bR^n)$. The reader can refer to \cite{ChungKim1994_Nonunique-heat} for a non-identically zero solution $u \in C^\infty(\bR^2_+) \cap C([0,\infty) \times \bR)$ with $u(0,x)=0$ everywhere and
    $ |u(t,x)| \le C(\epsilon) e^{\epsilon/t} $
    for any $\epsilon>0$. The continuity implies uniform convergence of $u(t)$ to 0 on compact intervals as $t \to 0$, and hence convergence in distributional sense.
\end{remark} 

\subsection*{Acknowledgement}
The authors were supported by the ANR project RAGE ANR-18-CE40-0012. We would like to thank Ioann Vasilyev for telling us of Gushchin's results and Patrick G\'erard for stimulating discussions.
\section{Proof of the theorem}

Our main lemma asserts the $L^2$-growth on rectangles of caloric function implies internal semigroup representation, which is an interesting fact in its own sake that we did not find in the literature. It can be seen as a particular case of \cite[Thm.~5.1]{Auscher-Monniaux-Portal2019existence} obtained for general parabolic equations with time-dependent, bounded measurable and elliptic coefficients. For the heat equation,  an elementary proof based on Green's formula and basic estimates will be given.

\begin{lemma}
    \label{lemma:homo-id}
    Let $u \in \scrD'((a,b) \times \bR^n)$ be a distributional solution to the heat equation. Suppose that \eqref{e:L2condition} holds. Then when $a<s<t<b$, for any $h \in \Cc(\bR^n)$,
    \begin{equation}
        \int_{\bR^n} u(t,x) h(x) dx = \int_{\bR^n} u(s,x) (e^{(t-s)\Delta} h)(x) dx
        \label{e:homo-id}
    \end{equation}
    where the right-hand integral absolutely converges.
\end{lemma}

The identity \eqref{e:homo-id} can  heuristically be called ``\textit{homotopy identity}", as it  formally 
shows $u(t)=e^{(t-s)\Delta} u(s)$ in the sense of distributions. In fact, once \eqref{e:homo-id} is shown, one can extend it to all $h \in \scrS(\bR^n)$, so that this holds in the sense of tempered distributions. 

\begin{proof}
    For $-\infty<\tau\le t$, define 
    \[ \varphi(\tau,x):=(e^{(t-\tau)\Delta} h)(x), \]
    which satisfies $\partial_\tau \varphi + \Delta \varphi = 0$ on $(-\infty,t)\times \bR^n$. Green's formula implies for any $r>0$,
    \begin{equation}
        \int_{B(0,r)} (u \Delta \varphi - \varphi \Delta u)(\tau) = \int_{\partial B(0,r)} (u \nabla \varphi - \varphi \nabla u)(\tau) \cdot \bfn \, d\sigma,
        \label{e:hi-green}
    \end{equation}
    where $\bfn$ is the outer unit normal vector and $d\sigma$ is the sphere volume form. Here and in the sequel, unspecified measures are Lebesgue measures. Newton-Leibniz formula yields
    \begin{equation}
        \int_s^t \int_{B(0,r)} (\varphi \partial_\tau u + u \partial_\tau \varphi) = \int_{B(0,r)} (u(t)\varphi(t)-u(s)\varphi(s)).
        \label{e:hi-NL}
    \end{equation}
    Integrating \eqref{e:hi-green} over $s<\tau<t$ and adding it to \eqref{e:hi-NL}, we have
    \begin{equation}
        \int_{B(0,r)} u(t) \varphi(t) - \int_{B(0,r)} u(s) \varphi(s) = \int_s^t \int_{\partial B(0,r)} (\varphi \nabla u - u \nabla \varphi) d\sigma d\tau.
        \label{e:hi-on-B(0,r)}
    \end{equation}
    Note that $u(t)\varphi(t) \in L^1(\bR^n)$ as $\varphi(t)=h$ has compact support. We claim $u(s)\varphi(s)$ also lies in $L^1(\bR^n)$. Indeed, pick $\rho>0$ such that $\supp(h) \subset B(0,\rho)$. Let $\kappa>1$ be a constant to be determined. Denote by $C_0$ the ball $B(0,\kappa \rho)$ and by $C_j$ the annulus $\{x \in \bR^n:\kappa^j \rho \le |x| < \kappa^{j+1} \rho\}$ for $j\ge 1$. We have
    \begin{equation}
        \int_{\bR^n} |u(s)| |e^{(t-s)\Delta} h|  \le \sum_{j \ge 0} \|u(s)\|_{L^2(C_j)} \| e^{(t-s)\Delta} h \|_{L^2(C_j)}.
        \label{e:hi-uvarphi(s)-L1}
    \end{equation}
    Note that only terms with $j \gg 1$ are of concern as both $u(s)$ and $\varphi(s)$ are bounded on compact sets. Caccioppoli's inequality\footnote{By this we mean the energy estimates obtained by multiplying $u$ with proper cut-off and seeing this as a solution to the heat equation with localised source term.}
    and \eqref{item:cond_size} imply
    \begin{align*}
        \|u(s)\|_{L^2(C_j)}^2 
        &\lesssim_\kappa \left( \frac{1}{(\kappa^{j+1} \rho)^2} + \frac{1}{s-a} \right) \int_a^b \|u(\tau)\|_{L^2(B(x,\kappa^{j+2} \rho))}^2 d\tau \\
        &\lesssim_{\rho,a,b} \frac{1}{s-a} \exp\left( \frac{2\gamma}{b-a} (\kappa^{j+2} \rho)^2 \right).
    \end{align*}
    The heat kernel representation implies
    \[ \|e^{(t-s)\Delta} h\|_{L^2(C_j)} \lesssim \exp\left( -\frac{cd_j^2}{t-s} \right) \|h\|_2 \]
    for $0<c<1/4$ and $d_j := \dist(C_j,\supp(h))$, which asymptotically equals to $\kappa^j \rho$. Pick $c$ close to $1/4$ and $\kappa$ close to 1 so that $\gamma<c/\kappa^4$, and thus for $j \gg 1$,
    \[ \gamma < \frac{c d_j^2}{(\kappa^{j+2} \rho)^2} < \frac{c d_j^2}{(\kappa^{j+2} \rho)^2} \cdot \frac{b-a}{t-s}. \]
    It ensures that the sum in \eqref{e:hi-uvarphi(s)-L1} converges and the claim hence follows.

    Thus, it suffices to prove that there exists an increasing sequence $(r_m)$ tending to $\infty$ such that
    \[ \lim_{m \rightarrow +\infty} \int_s^t \int_{\partial B(0,r_m)} (|\varphi \nabla u| + |u \nabla \varphi|)\,  d\tau d\sigma = 0. \]
    Indeed, suppose so, and then applying \eqref{e:hi-on-B(0,r)} for $(r_m)$ and taking limits on $m$ imply \eqref{e:homo-id} holds. Let us show the existence of such sequence. Let $0<\lambda<1$ be a constant to be determined. Define
    \begin{align*}
        \Phi(R)
        &:=\int_{\lambda R}^R r^{n-1} \int_s^t \int_{\partial B(0,r)} (|\varphi \nabla u| + |u \nabla \varphi|)\, d\tau d\sigma  dr\\
        &= \int_s^t \int_{\lambda R<|x|<R} (|\varphi \nabla u| + |u \nabla \varphi|)\, d\tau dx,
    \end{align*}
    and denote by $\Phi_i(R)$ the $i$-th term for $i=1,2$. Then, it is enough to show that $\Phi(R)$ is bounded. For $\Phi_1(R)$, Caccioppoli's inequality, kernel representation of the heat semigroup, and \eqref{item:cond_size} altogether imply
    \begin{align*}
        \Phi_1(R) 
        &\le \left( \int_s^t \int_{\lambda R<|x|<R} |\nabla u|^2 \right)^{1/2} \left( \int_s^t \int_{\lambda R<|x|<R} |\varphi|^2 \right)^{1/2} \\
        &\lesssim_\kappa \left [ \frac{1}{s-a} \left ( 1+\frac{t-a}{R^2} \right ) \int_a^b \int_{|x|<\kappa R} |u|^2 \right ]^{1/2} \left( \int_s^t e^{-\frac{2cd^2}{t-\tau}} d\tau \right )^{1/2} \|h\|_2 \\
        &\lesssim_{a,b,\kappa} \left( \frac{t-s}{s-a} \right)^{1/2} \exp \left( -\frac{cd^2}{t-s}+\frac{\gamma (\kappa R)^2}{b-a} \right) \|h\|_2
    \end{align*}
    for sufficiently large $R$. Here, $\kappa>1$ and $0<c<1/4$ are constants to be determined, and $d:=\dist(\supp(h),\{\lambda R<|x|<R\})$. Then, pick $\kappa$ close to 1, $\lambda$ close to 1, and $c$ close to $1/4$ so that $\gamma<c\lambda^2 /\kappa^2$, and thus
    \begin{equation}
        \gamma < \frac{cd^2}{(\kappa R)^2} < \frac{cd^2}{(\kappa R)^2} \cdot \frac{b-a}{t-s}
        \label{e:hi_cond_gamma}
    \end{equation}
    for sufficiently large $R$. We hence conclude that $\Phi_1(R)$ is bounded.

    For $\Phi_2(R)$, the kernel representation of $(t^{1/2} \nabla e^{t\Delta})_{t>0}$ and \eqref{item:cond_size} yield
    \begin{align*}
        \Phi_2(R) 
        &\le \left( \int_a^b \int_{\lambda R<|x|<R} |u|^2 \right)^{1/2} \left( \int_s^t \int_{\lambda R<|x|<R} |\nabla \varphi|^2 \right)^{1/2} \\
        &\lesssim e^{\frac{\gamma R^2}{b-a}} \left( \int_s^t e^{-\frac{2cd^2}{t-\tau}} \frac{d\tau}{t-\tau} \right)^{1/2} \|h\|_2 \\
        &\lesssim \frac{(t-s)^{1/2}}{c^{1/2} d} \exp \left( -\frac{cd^2}{t-s}+\frac{\gamma R^2}{b-a} \right) \|h\|_2.
    \end{align*}
    The same choice for $\lambda$ and $c$ as above implies \eqref{e:hi_cond_gamma}, and hence the boundedness of $\Phi_2(R)$. This completes the proof.
\end{proof}

Let us recall some topological facts about tempered distributions.

\begin{lemma}
    \label{lemma:BS_Frechet}
    Let $X$ be a Fr\'echet space and $Y$ be a normed space. Let $I$ be an index set and $\{\psi_\alpha\}_{\alpha \in I}$ be a collection of continuous linear maps from $X$ to $Y$. If $\sup_{\alpha \in I} \|\psi_\alpha(x)\|_Y$ is bounded for any $x \in X$, then the family $\{\psi_\alpha\}_{\alpha \in I}$ is equicontinuous.
\end{lemma}

\begin{proof}
    It is a direct consequence of a generalized version of Banach-Steinhaus theorem on barrelled spaces, due to the fact that Fr\'echet spaces are barrelled spaces, see \cite[\S III.4]{Bourbaki2003TVS}.
\end{proof}

One can easily obtain two corollaries. The pairings below are all understood in the sense of tempered distributions.

\begin{coro}
    \label{cor:cv_pair_S}
    Let $(\varphi_k)$ be a sequence converging to $\varphi$ in $\scrS(\bR^n)$ and $(u_k)$ be a sequence converging to $u$ in $\scrS'(\bR^n)$. Then $(\langle u_k,\varphi_k \rangle)$ converges to $\langle u,\varphi \rangle$.
\end{coro}

\begin{coro}
    \label{cor:compact_S'}
    Any bounded sequence in $\scrS'(\bR^n)$ has a convergent subsequence.
\end{coro}

Let us provide the proof of Theorem \ref{thm:rep}.
\begin{proof}[Proof of Theorem \ref{thm:rep}]
    Given $0<s<t<T$, pick $a,b$ so that $0<a<s<t<b<T$. Applying
    Lemma \ref{lemma:homo-id}, we get for any $h \in \Cc(\bR^n)$,
    \begin{equation}
        \int_{\bR^n} u(t,x) h(x) dx = \int_{\bR^n} u(s,x) (e^{(t-s)\Delta} h)(x) dx.
        \label{e:hi}
    \end{equation}
    Moreover, Corollary \ref{cor:compact_S'} implies there is a subsequence $(t_{k_j})$ so that $(u(t_{k_j}))$ converges to some $u_0$ in $\scrS'(\bR^n)$ as $j\to \infty$. One can also easily verify that $e^{(t-s)\Delta} h$ converges to $e^{t\Delta} h$ in $\scrS(\bR^n)$ as $s \to 0$. Thus, we infer from Corollary \ref{cor:cv_pair_S} that
    \[ \lim_{j \rightarrow \infty} \left \langle u(t_{k_j}), e^{(t-t_{k_j})\Delta} h \right \rangle = \left \langle u_0,e^{t\Delta} h \right \rangle. \]
    Applying \eqref{e:hi} for $s=t_{k_j}$ and taking limits on $j$ yield $u(t)=e^{t\Delta} u_0$ in $\scrD'(\bR^n)$. As the right-hand side belongs to  $\scrS'(\bR^n)$, so does $u(t)$ for all $0<t<T$. In particular, it implies $u(t)$ converges to $u_0$ in $\scrS'(\bR^n)$ as $t \to 0$, so $u_0$ is unique. This completes the proof.
\end{proof}
\section{Applications}
\label{sec:applications}

Our statement covers many functional spaces of common use in analysis, even when one expects that the initial data is a distribution. To keep this note short, let us illustrate the result with an example related to the famous work of H. Koch and D. Tataru on Navier-Stokes equations \cite{Koch-Tataru2001NS}. More applications will be studied in forthcoming works. Define the \textit{tent space} $T^\infty$ as the collection of measurable functions $u$ for which
\[ \|u\|_{T^\infty} := \sup_B \left( \frac{1}{|B|}\int_0^{r(B)^2} \int_B |u(t,y)|^2 dtdy \right)^{1/2} < \infty, \]
where $B$ describes  balls in $\bR^n$ and $|B|$ is the Lebesgue measure of $B$. Let $\bmo^{-1}$ be the collection of distributions $f \in \scrD'(\bR^n)$ with $f=\Div g$ for some $g \in \bmo(\bR^n;\bC^n)$. The space $\bmo^{-1}$ can be embedded into $\scrS'(\bR^n)$, and identifies with the homogeneous Triebel-Lizorkin space $\DotF^{-1}_{\infty,2}$, with equivalent norms, see, \textit{e.g.}, \cite[\S 5.1]{Triebel1983_Space}. Moreover, a well-known characterisation of $\bmo^{-1}$ is that 
\begin{equation}
    \|f\|_{\bmo^{-1}} \eqsim \|e^{t\Delta} f\|_{T^\infty}.
    \label{e:bmo-1_char}
\end{equation}
In particular, a tempered distribution $f$ lies in $\bmo^{-1}$ if $(e^{t\Delta} f)(x)$ lies in $T^\infty$. Remark that if $f$ belongs to $\bmo^{-1}$,  as a function of $t\in [0,\infty)$, $e^{t\Delta} f$ is continuous in $\bmo^{-1}$ equipped with its weak star topology (or equipped with the topology inherited from that of $\scrS'(\bR^n)$, by density).

A natural question is whether all $T^\infty$ functions solving the heat equation are of that form. We answer it in the affirmative.
\begin{theorem}
    \label{thm:BMO-1}
    Given a global distributional solution to the heat equation $u \in T^\infty$, there exists a unique $u_0 \in \bmo^{-1}$ so that $u(t)=e^{t\Delta} u_0$ for any $t>0$.
\end{theorem}

\begin{proof} 
    It suffices to verify the two conditions in Theorem \ref{thm:rep}. Indeed, it shows that there exists a unique $u_0 \in \scrS'(\bR^n)$ so that $u(t)=e^{t\Delta} u_0$ for any $t>0$. We then get $u_0 \in \bmo^{-1}$ by \eqref{e:bmo-1_char} since $ (e^{t\Delta} u_0)(x)=u(t,x)$ belongs to $T^\infty$.

    Let us verify the conditions. First, \eqref{item:cond_size} readily follows as for $0<a<b<\infty$,
    \begin{equation}
        F(x) := \left( \int_a^b \int_{B(x,b^{1/2})} |u(t,y)|^2 dtdy \right)^{1/2}
        \label{e:F-in-Linfty}
    \end{equation}
    satisfies $\|F\|_{L^\infty(\bR^n)} \lesssim \|u\|_{T^\infty}$. 
    
    Next, we claim  that there exists $M>0$ so that for any $\varphi \in C_c^\infty(\bR^n)$, 
    \begin{equation}
        \sup_{0<t<1/2} |\langle u(t),\varphi \rangle| \lesssim \cP_M(\varphi) \|u\|_{T^\infty},
        \label{e:bdd_uS'_by_uT}
    \end{equation}
    where $\cP_M$ is the semi-norm given by
    \[ \cP_M(\varphi) := \sup_{|\alpha|+|\beta| \le M} \sup_{x \in \bR^n} |x^\alpha \partial^\beta \varphi(x)|. \]
    For fixed $0<t<1/2$, standard considerations allow one to extend $u(t)$  to a tempered distribution so that \eqref{e:bdd_uS'_by_uT} holds 
    for all $\varphi \in \scrS(\bR^n)$, which proves \eqref{item:cond_initial}. As for the claim, fix $0<t<1/2<t'<1$ and let $\varphi \in C_c^\infty(\bR^n)$. Using the equation for $u$ and integration by parts,
    \begin{align*}
    |\langle u(t'),\varphi \rangle - \langle u(t),\varphi \rangle|
     &\le \int_t^{t'} \int_{\bR^n} |u(s,x)| |\Delta \varphi(x)| dsdx \\ 
     &= \int_t^{t'} \int_{B(0,1)} |u(s,x)| |\Delta \varphi(x)| dsdx \\ 
     &\quad + \sum_{k=1}^\infty \int_t^{t'} \int_{2^{k-1} \le |x| < 2^k} |u(s,x)| |\Delta \varphi(x)| dsdx.
    \end{align*}
    Denote by $I_0$ the first term and $I_k$ the $k$-th term in the summation. Cauchy-Schwarz inequality yields
    \begin{align*}
    I_0
     &\le |B(0,1)| (\sup_{|x|<1} |\Delta \varphi(x)|) \left ( \frac{1}{|B(0,1)|}\int_0^1 \int_{B(0,1)} |u(s,x)|^2 dsdx \right )^{1/2} \\ 
     &\lesssim_n \cP_2(\varphi) \|u\|_{T^\infty}.
    \end{align*}
    Similarly, we also have that
    \begin{align*}
    I_k
     &\lesssim_n 2^{kn} (\sup_{2^{k-1} \le |x| < 2^k} |\Delta \varphi(x)|) \left ( \frac{1}{|B(0,2^k)|}\int_0^1 \int_{B(0,2^k)} |u(s,x)|^2 dsdx \right )^{1/2} \\ 
     &\lesssim_n 2^{-k} ( \sup_{x \in \bR^n} |x|^{n+1} |\Delta \varphi(x)| ) \|u\|_{T^\infty} \le 2^{-k} \cP_{n+3}(\varphi) \|u\|_{T^\infty}.
    \end{align*}
    We thus obtain 
    \[ |\langle u(t'),\varphi \rangle - \langle u(t),\varphi \rangle| \lesssim_n \cP_{n+3}(\varphi) \|u\|_{T^\infty}. \]
    Taking average in  $t'\in (1/2,1)$ implies
    \[ |\langle u(t),\varphi \rangle| \lesssim_n \int_{1/2}^1 |\langle u(t'),\varphi \rangle| dt' + \cP_{n+3}(\varphi) \|u\|_{T^\infty}. \]
    The same argument as above yields
    \[ \int_{1/2}^1 |\langle u(t'),\varphi \rangle| dt' \le \int_{1/2}^1 \int_{\bR^n} |u(t',x)| |\varphi(x)| dt'dx \lesssim_n \cP_{n+1}(\varphi) \|u\|_{T^\infty}. \]
    This completes the proof. 
\end{proof}

\bibliographystyle{alpha}
\bibliography{sample}

\end{document}